\documentclass{article} 
\usepackage{amssymb,amsmath,amsthm}
\usepackage{amscd}
\usepackage{hyperref}
\usepackage{epsfig}
\usepackage{amsfonts}
\usepackage{amssymb}
\usepackage{amsmath,enumerate}
\usepackage{commath}
\usepackage{euscript}
\usepackage{enumitem}
\usepackage{amscd}
\usepackage{xcolor}
\usepackage[ruled,vlined]{algorithm2e}
\usepackage[numbers,sort&compress]{natbib}

\newtheorem{theorem}{Theorem}[section]
\newtheorem{remark}{Remark}[section]
\newtheorem{corollary}{Corollary}[section]
\newtheorem{definition}{Definition}[section]

\newtheorem{lemma}{Lemma}[section]
\newtheorem{proposition}{Proposition}[section]

\begin{document}
	\title{Variational Reformulation of Generalized Nash Equilibrium Problems with Non-ordered Preferences}
	\author{
		Asrifa Sultana\footnotemark[1] \footnotemark[2] , Shivani Valecha\footnotemark[2] }
	\date{ }
	\maketitle
	\def\thefootnote{\fnsymbol{footnote}}
	
	\footnotetext[1]{ Corresponding author. e-mail- {\tt asrifa@iitbhilai.ac.in}}
	\noindent
	\footnotetext[2]{Department of Mathematics, Indian Institute of Technology Bhilai, Raipur - 492015, India.
	}
	
	
		
	\begin{abstract}\noindent
		The preferences of players in non-cooperative games represent their choice in the set of available options, which meet the completeness property if players are able to compare any pair of available options. In the existing literature, the variational reformulation of generalized Nash games relies on the numerical representation of players' preferences into objective functions, which is only possible if preferences are transitive and complete. 
		In this work, 
		we first characterize the jointly convex generalized Nash equilibrium problems in terms of a variational inequality without requiring any numerical representation of preferences. Furthermore, we provide the suitable conditions under which any solution of a quasi-variational inequality becomes an equilibrium for the generalized Nash game with non-ordered (incomplete and non-transitive) non-convex inter-dependent preferences. 
		We check the solvability of these games when the strategy maps of players are (possibly) unbounded. As an application, we derive the occurrence of competitive equilibrium for Arrow-Debreu economy under uncertainty.
	\end{abstract}
	
	\section{Introduction}
	\hspace{1em}
	The preference of any player in non-cooperative games represents his choice in the set of available options and it mostly depends on other players' choices \cite{shafer}. According to Debreu \cite{debreubook,debreu1}, the preference of any player is representable in terms of real-valued objective functions if it meets certain conditions, namely, transitivity and completeness. In game theory, one can maximize the corresponding objective functions in place of maximizing preferences of players if the mentioned conditions are fulfilled. In particular, the following form of generalized Nash equilibrium problem (GNEP) is well-known in literature, where each player intends to maximize a real-valued objective function representing his preference \cite{debreu,faccsurvey}. Consider a set $\Lambda=\{1,2,\cdots N\}$ consisting of $N$-players. Suppose any player $i\in \mathcal \Lambda$ regulates a strategy variable $x_i\in \mathbb{R}^{n_i}$ where $\sum_{i\in \Lambda} n_i=n$. Then, we denote $x=(x_i)_{i\in \Lambda}\in \mathbb{R}^n$ as $x=(x_{-i},x_i)$ where $x_{-i}$ is the strategy vector of all the players except player $i$. The feasible strategy set of each player is given as $K_i(x)\subseteq \mathbb{R}^{n_i}$, which depends on his strategy and the strategies chosen by rivals. For a given strategy $x_{-i}$ of rivals, any player $i$ intends to choose a strategy $x_i\in K_i(x)$ such that $x_i$ maximizes his objective function $u_i:\mathbb{R}^n\rightarrow \mathbb{R}$,
	\begin{equation}\label{GNEP2}
		u_i(x_{-i},x_i) = \max_{y_i\in K_i(x)} u_i(x_{-i}, y_i).
	\end{equation}
	Suppose $Sol_i(x_{-i})$ consists of all such vectors $x_i$, which solves the problem (\ref{GNEP2}). Then a vector $\bar x$ is known as generalized Nash equilibrium if $\bar x_i\in Sol_i(\bar x_{-i})$ for each $i\in \mathcal{I}$. For a given $x=(x_{-i},x_i)\in \mathbb{R}^n$, if the feasible strategy set $K_i(x)$ is particularly defined as
	\begin{equation}\label{Kiconv}
		K_i(x)=\{y_i\in\mathbb{R}^{n_i}|\,(x_{-i},y_i)\in \mathcal X\},
	\end{equation} 
where $\mathcal X\subseteq\mathbb{R}^n$ is a non-empty closed convex set. Then, this type of GNEP is classified as jointly convex GNEP \cite{rosen,facc,ausselGNEP}.
	
	The variational inequality (VI) theory introduced by Stampacchia \cite{stampacchia}, works as an efficient tool to study optimization problems and Nash games \cite{ausselnormal,ausselGNEP,facc,milasipref2021}. 
	It is well known that solving a generalized Nash equilibrium problem with concave differentiable objective functions is equivalent to solving a quasi-variational inequality (QVI) where the operator is formed by using gradients of all objective functions \cite{faccsurvey}. Further, one can determine some solutions of jointly convex GNEP having concave differentiable objective functions by solving a VI problem, in which the operator is formed by using gradients of all objective functions \cite{facc}. In the case of jointly convex GNEP with quasi-concave objective functions, the operator for an associated VI problem is formed by using normal cones to super-level sets corresponding to the objective functions \cite{ausselGNEP,cotrina}. This characterization of jointly convex GNEP through VI problems is helpful in studying several complex economic problems, for instance, time-dependent electricity market model \cite{aussel-rachna}. In the existing literature, the variational reformulation of GNEP relies on the numerical representation of players' preferences into objective functions, which is only possible if preferences are transitive and complete \cite{debreubook}. 
	
	In fact, preference of any player meets the completeness property if he is able to compare any pair of available options. As per \cite{neumann}, the completeness property for preferences is not always valid in real-world scenarios. This motivated several authors like Shafer-Sonnenschein \cite{shafer}, Tian \cite{tian}, He-Yannelis \cite{yannelis} to study GNEP with non-ordered (incomplete and non-transitive) inter-dependent preferences by using fixed point theory. On the other hand, Bade \cite{bade} studied Nash games with incomplete preferences by using scalarization approach. Recently, Milasi-Scopelliti \cite{milasipref2021} and Milasi-Puglis-Vitanza \cite{milasipref} 
	studied maximization problem for incomplete preference relations by using variational approach and as an application they derived existence results for economic equilibrium problems under uncertainty. However, there is no detailed study on the variational reformulation of a GNEP in which players have non-ordered inter-dependent preferences. 
	
	In this work, we aim to study variational reformulation for the generalized Nash equilibrium problems in which the players have non-ordered non-convex inter-dependent preferences (refer \cite{shafer,tian,yannelis} for more details on such games). In particular, we characterize the jointly convex GNEP in terms of a variational inequality without requiring any numerical representation of preferences. In this regard, the adapted normal cones corresponding to preferences are used by us to form an operator for the variational inequality associated with the jointly convex GNEP. Furthermore, we provide the suitable conditions under which any solution of a quasi-variational inequality becomes an equilibrium for the generalized Nash games considered in \cite{shafer,tian,yannelis}. We check the solvability of these games when the strategy maps of players are (possibly) unbounded by using variational technique. Finally, we apply our results to ensure the occurrence of competitive equilibrium for Arrow-Debreu economy under uncertainty.
	
	\section{Preliminaries}
	Suppose $C$ is a subset of $\mathbb{R}^m$, then the polar set of $C$ denoted by $C^{\circ}$ is generally defined as \cite{ausselnormal,cotrina}, 
	$$ C^{\circ}= \{y\in \mathbb{R}^m|\,\langle y,x\rangle\leq 0~\text{for all}~x\in C\}.$$
	Furthermore, the normal cone of the set $C$ at some point $x\in \mathbb{R}^m$ is given as,
	$$N_C(x)=(C-\{x\})^\circ=\{x^*\in \mathbb{R}^m|\,\langle x^*,y-x\rangle \leq 0~\text{for all}~y\in C\}.$$ 
	We assume $N_C(x)=\mathbb{R}^m$ whenever $C=\emptyset$. 
	
	For a given set $C\subseteq \mathbb{R}^m$, we denote $co(C)$ and $cl(C)$ as convex hull and closure of the set $C$, respectively. Moreover, $P:C\rightrightarrows \mathbb{R}^p$ denotes a multi-valued map, that is, $P(x)\subseteq \mathbb{R}^p$ for any $x\in C$. The map $P$ is said to admit open upper sections \cite{tian} if $P(x)$ is open for any $x\in C$. The reader may refer to \cite{aubin}, in order to recall some important concepts of upper semi-continuous (u.s.c.), lower semi-continuous (l.s.c.) and closed multi-valued maps. 
	
	Let us recollect the definitions of variational and quasi-variational inequality problems. Suppose $K\subseteq \mathbb{R}^n$ and $T:K\rightrightarrows \mathbb{R}^n$ is a multi-valued map. Then, the Stampacchia variational inequality problem $VI(T,K)$ \cite{aussel-hadj} corresponds to determine $\bar x\in K$ such that,
	$$\text{there exists}~\bar x^*\in T(\bar x)~\text{satisfying}~\langle \bar x^*,z-\bar x\rangle\geq 0~\text{for all}~z\in K.$$
	Suppose $T:D\rightrightarrows \mathcal{R}^n$ and $K:D\rightrightarrows D$ are multi-valued maps where $D\subseteq \mathbb{R}^n$. Then, the quasi-variational inequality problem $QVI(T,K)$ \cite{chan,ausselcoer} corresponds to determine $\bar x\in K(\bar x)$ such that,
	$$\text{there exists}~\bar x^*\in T(\bar x)~\text{satisfying}~\langle \bar x^*,z-\bar x\rangle\geq 0~\text{for all}~z\in K(\bar x).$$
	
	For given $Y_1\subseteq\mathbb{R}^m$ and $Y_2\subseteq\mathbb{R}^p$, suppose $P:Y_1\times Y_2\rightrightarrows Y_2$ is a multi-valued map. 
	Let us define a multi-valued map $\mathcal N_{P}:Y_1\times Y_2\rightrightarrows \mathbb{R}^p$ corresponding to the map $P$ as follows \cite{ausselGNEP,cotrina},
	\begin{equation}\label{normal}
		\mathcal N_{P}(x,y)= N_{P(x,y)}(y)= \{y^*\in \mathbb{R}^p|\,\langle y^*,z-y\rangle \leq 0~\forall~z\in P(x,y)\}.
	\end{equation}
	
	We require the following results in order to prove the upcoming existence results on generalized Nash games defined by preference relations, 
	\begin{proposition} \label{closed}
		Consider a map $P:Y_1\times Y_2\rightrightarrows Y_2$ where $Y_1$ and $Y_2$ are subsets of $\mathbb{R}^m$ and $\mathbb{R}^p$, respectively. Then, the map $\mathcal N_{P}:Y_1\times Y_2\rightrightarrows \mathbb{R}^p$ defined as (\ref{normal}) meets following properties,
		\begin{itemize}
			\item[(a)] $\mathcal N_{P}(x,y)\setminus \{0\}\neq \emptyset$ if $P(x,y)$ is a convex subset of $Y_2$ satisfying $y\notin P(x,y)$ for any $(x,y)\in Y_1\times Y_2$;
			\item[(b)] $\mathcal N_{P}$ is a closed map if $P:Y_1\times Y_2\rightrightarrows Y_2$ is lower semi-continuous.
		\end{itemize}
	\end{proposition}
	\begin{proof}
		\begin{itemize}
			\item[(a)] Suppose $(x,y)\in Y_1\times Y_2$ is arbitrary. If $P(x,y)=\emptyset$, then our claim follows trivially as $\mathcal N_{P}(x,y)=\mathbb {R}^p$. Suppose $P(x,y)\neq \emptyset$ then in the view of assumption (a), we obtain some $0\neq y^*\in \mathbb R^p$ by applying separation theorem \cite[Section 2.5.1]{boyd} such that, $$\langle y^*,z\rangle \leq \langle y^*,y\rangle~\text{for all}~z\in P(x,y).$$ This, finally leads us to the conclusion $y^*\in \mathcal N_{P}(x,y)\setminus \{0\}$.
			\item[(b)]In order to prove that the map $\mathcal N_{P}$ is a closed map we consider sequences $\{(x_n,y_n)\}_{n\in \mathbb{N}}\in Y_1\times Y_2$ and $y_n^*\in  \mathcal N_{P}(x_n, y_n)~\forall~n\in \mathbb{N}$ such that $(x_n,y_n)\rightarrow(x,y)$ and $y_n^*\rightarrow y^*$. We aim to show that $y^*\in \mathcal N_{P}(x,y)$. In fact, the claim follows trivially if $P(x,y)=\emptyset$. Suppose $z\in P(x,y)$, then by employing the lower semi-continuity of $P$ we obtain some sequence $z_n\in P(x_n,y_n)$ such that $z_n\rightarrow z$. Since, we already have $\{y_n^*\}\in  \mathcal N_{P}(x_n, y_n)$, we obtain
			$$ \langle y_n^*, z_n-y_n\rangle \leq 0,\quad \forall n\in \mathbb{N}.$$ 
			Finally, we obtain $\langle y^*,z-y\rangle \leq 0$ by taking $n\rightarrow \infty$. Hence, $y^*\in \mathcal N_{P}(x,y)$ and $\mathcal N_{P}$ becomes a closed map.
		\end{itemize}
	\end{proof}
	
	\begin{proposition}\label{relation}
		Consider a map $P:Y_1\times Y_2\rightrightarrows Y_2$ where $Y_1$ and $Y_2$ are subsets of $\mathbb{R}^m$ and $\mathbb{R}^p$, respectively. Suppose $P(x,y)$ is an open convex subset of $Y_2$ for any $(x,y)\in Y_1\times Y_2$. Then, $\langle y^*,z-y\rangle<0$ for any $z\in P(x,y)$ and $0\neq y^*\in \mathcal N_{P}(x,y)$. 
	\end{proposition}
	\begin{proof}
		Suppose $z\in P(x,y)$ then $z-y\in (P(x,y)-\{y\})=int (P(x,y)-\{y\})$. In the view of \cite[Lemma 2.1]{cotrina}, we obtain $\langle y^*,z-y\rangle<0$ for any $y^*\in (P(x,y)-\{y\})^\circ=\mathcal N_{P}(x,y)$ with $y^*\neq 0$. 
	\end{proof}
	
	\section{Generalized Nash Games with Non-ordered Inter-dependent Preferences} \label{mainresult}
	The concept of generalized Nash equilibrium problems with non-ordered inter-dependent preferences was introduced by Shafer-Sonnenschein \cite{shafer}. They investigated the occurrence of equilibrium for such problems by using celebrated Kakutani's fixed point result. Later, several other researchers (like \cite{tian,yannelis}) derived the occurrence of equilibrium for such games by relaxing the sufficient conditions required in \cite{shafer}. 
	
	Let us recall the definition of GNEP with non-ordered inter-dependent preferences \cite{shafer,yannelis,tian}. Suppose $\Lambda=\{1,2,\cdots N\}$ denotes the set of involved agents. Let $X_i\subseteq \mathbb{R}^{n_i}$ denotes the choice set for each agent $i\in \Lambda$, where $\sum_{i\in \Lambda}n_i=n$. Suppose that 
   $$X=\prod_{i\in \Lambda} X_i \subseteq \mathbb{R}^n~\text{and}~ X_{-i}= \prod_{j\in (\Lambda\setminus \{i\})} X_j \subseteq \mathbb{R}^{n-n_i}.$$ 
	We consider the multi-valued maps $P_i:X\rightrightarrows X_i$ and $K_i:X\rightrightarrows X_i$ as preference map and constraint map, respectively, for any $i\in \Lambda$. Let us represent the generalized Nash game defined by preference maps, also known as \textit{abstract economy}, by N-ordered triples $\Gamma=(X_i,K_i,P_i)_{i\in \Lambda}$. A vector $\tilde x\in X$ is said to be equilibrium for the game $\Gamma$ if $\tilde x\in K_i(\tilde x)$ and $P_i(\tilde x)\cap K_i(\tilde x)=\emptyset$ for each $i\in \Lambda$.
	
	\begin{remark}\label{numGNEP}
		If the preference map $P_i:X_{-i}\times X_i\rightrightarrows X_i$ is representable by utility function $u_i:X_{-i}\times X_i\rightarrow \mathbb{R}$ then $P_i(x_{-i},x_i)=\{y_i\in X_i\,|\, u_i(x_{-i},y_i)>u_i(x_{-i},x_i)\}$ (see \cite{shafer}) and we observe the given game $\Gamma$ reduces to Arrow-Debreu abstract economy \cite{debreu}.
	\end{remark}
	\subsection{Jointly Convex GNEP with Preference maps and Variational Inequalities} \label{GNEP}
	In the seminal work by Rosen \cite{rosen}, a special class of GNEP known as \textit{jointly convex} GNEP or generalized Nash equilibrium problem with \textit{coupled constraints} was initiated. In past few years, Facchinei et al. \cite{facc}, Aussel-Dutta \cite{ausselGNEP} and Bueno-Cotrina \cite{cotrina} characterized the jointly convex GNEP in terms of a variational inequality under concavity, semi-strict quasi-concavity and quasi-concavity of objective functions, respectively. In this section, we characterize the jointly convex GNEP in terms of a variational inequality without requiring any numerical representation of preferences into objective functions. Thus, we consider that players have non-ordered inter-dependent preferences. Further, we provide the sufficient conditions required for occurrence of equilibrium in such games by using this characterization.
	
	
	\textcolor{black}{According to \cite{rosen,facc,ausselGNEP}, the feasible strategy set for any player $i\in \Lambda$ is defined as (\ref{Kiconv}) in a jointly convex GNEP. According to (\ref{Kiconv}), one can observe that the choice set $X_i$ for any player is actually the projection of the set $\mathcal{X}$ on $\mathbb{R}^{n_i}$ (see \cite{rosen,cotrina} for more details). In a given abstract economy $\Gamma=(X_i, K_i, P_i)_{i\in \Lambda}$, let us consider the case where the constraint map $K_i:X\rightrightarrows X_i$ is defined in terms of non-empty, convex and closed set $\mathcal X \subset \mathbb{R}^n$,
	\begin{equation}\label{constraint}
		K_i(x)=\{z_i\in X_i|\,(x_{-i}, z_i)\in \mathcal X\}.
	\end{equation} 
	 If the map $K_i$ is defined as (\ref{constraint}) for each $i\in \Lambda$ the game $\Gamma=(X_i, K_i, P_i)_{i\in\Lambda}$ reduces to a jointly convex GNEP denoted by $\Gamma'=(\mathcal{X}, K_i, P_i)_{i\in \Lambda}$.} Note that the constraint maps $K_i$ are actually coupled due to joint constraint set $\mathcal X$.  
	
	Let us define the map $\tilde{P}_i:X_{-i}\times X_i\rightrightarrows \mathbb{R}^{n_i}$ as $\tilde P_i(x)= co(P_i(x))$. In order to construct a variational inequality corresponding to the game $\Gamma$, we define a map $T_i:X_{-i}\times X_i\rightrightarrows \mathbb{R}^{n_i}$ for each $i\in \Lambda$ as follows,
	\begin{equation} \label{Ti}
		T_i(x_{-i},x_i)=co(\mathcal{N}_{\tilde P_i}(x_{-i},x_i)\cap S_i[0,1]),
	\end{equation} 
	where $\mathcal{N}_{\tilde P_i}:X_{-i}\times X_i\rightrightarrows \mathbb{R}^{n_i}$ is defined as (\ref{normal}) and $S_i[0,1]=\{x\in \mathbb{R}^{n_i}|\,\norm{x}=1\}$. Further, suppose $T:X\rightrightarrows \mathbb{R}^n$ is defined as, 
	\begin{equation}
		T(x)=\prod_{i\in \Lambda}T_i(x).\label{T}
	\end{equation} 
	Then, clearly the map $T$ is convex and compact valued map. 
	
	Now, we provide the sufficient conditions on the preference relations $P_i$ under which the map $T$ becomes upper semi-continuous and non-empty valued.
	\begin{lemma} \label{usc}
		Suppose $T:X\rightrightarrows \mathbb{R}^n$ is a multi-valued map as defined in (\ref{T}). If for each $i\in \Lambda$,
		\begin{itemize}
			\item[(a)] $P_i$ is lower semi-continuous map then $T$ is an upper semi-continuous map;
			\item[(b)] $P_i$ satisfies $x_i\notin co (P_i(x))$ for any $x\in X$ then $T$ admits non-empty values.
		\end{itemize}
	\end{lemma}
	\begin{proof}
		\begin{itemize}
			\item[(a)] Suppose $i\in \Lambda$ is arbitrary. Then the map $\tilde P$ is lower semi-continuous according to \cite[Theorem 5.9]{rockafellar}. It is easy to observe that $\mathcal{N}_{\tilde P_i}:X_{-i}\times X_i\rightrightarrows \mathbb R^{n_i}$ is a closed map as per Proposition \ref{closed} (b). Hence, the map $T_i:X \rightrightarrows \mathbb{R}^{n_i}$ defined as (\ref{Ti}) is also closed. Finally, the fact that $T_i(x)\subseteq \bar B_i(0,1)$ for each $x\in X$ implies $T_i$ is u.s.c. map. This leads us to the conclusion $T=\prod_{i\in \Lambda} T_i$ is u.s.c. map.
			\item[(b)] Suppose $x\in X$ is arbitrary. According to Proposition \ref{closed}, the map $\mathcal{N}_{\tilde P_i}:X_{-i}\times X_i\rightrightarrows \mathbb R^{n_i}$ satisfies $\mathcal{N}_{\tilde P_i}(x_{-i},x_i)\setminus \{0\}\neq \emptyset$. Suppose $z_i^*\in \mathcal{N}_{\tilde P_i}(x_{-i},x_i)\setminus \{0\}$. Then, clearly $\frac{z_i^*}{\norm{z_i^*}}\in \mathcal{N}_{\tilde P_i}(x_{-i},x_i) \cap S_i[0,1]$. Hence, $T(x)=\prod_{i\in \Lambda} T_i(x)\neq \emptyset$ for any $x\in X$.
		\end{itemize}
	\end{proof}
	
	Now, we provide the sufficient conditions under which any solution of a variational inequality solves the jointly convex GNEP, $\Gamma'=(\mathcal{X}, K_i, P_i)_{i\in \Lambda}$.
	\begin{theorem}\label{VI}
		Assume that $\mathcal X\subseteq \mathbb{R}^n$ is non-empty and the map $P_i:X\rightrightarrows X_i$ has open upper sections for any $i\in \Lambda$. Suppose the map $T:X \rightrightarrows \mathbb{R}^n$ is defined as (\ref{T}) and the map $K_i:X \rightrightarrows X_i$ is defined as (\ref{constraint}). Then, any solution of $VI(T,\mathcal X)$ is an equilibrium for $\Gamma'=(\mathcal{X}, K_i, P_i)_{i\in \Lambda}$.
	\end{theorem}
	\begin{proof}
		Suppose $\tilde x\in \mathcal X$ solves $VI(T,\mathcal X)$ then,
		\begin{equation}\label{equ}
			\exists~\tilde x^*\in T(\tilde x), \langle \tilde x^*, y-\tilde x\rangle \geq 0~\text{for all}~y\in \mathcal X.
		\end{equation}
		Suppose $\tilde x$ is not a solution of $\Gamma$, then there exists $i\in \Lambda$ such that $P_i(\tilde x) \cap K_i(\tilde x)\neq \emptyset$. Suppose $z_i\in P_i(\tilde x) \cap K_i(\tilde x)$,
		then $y=(\tilde x_{-i},z_i)\in \mathcal X$ according to (\ref{constraint}). Hence, from (\ref{equ}) one can obtain,
		\begin{equation} \label{ineq}
			\langle \tilde x^*_i, z_i-\tilde x_i\rangle\geq 0.
		\end{equation}  
		Since $z_i\in \tilde P_i(\tilde x)$, we observe that $\tilde x_i^*=0$ by virtue of Proposition \ref{relation} as $\tilde x_i^*\in co(\mathcal N_{\tilde P_i}(\tilde x)\cap S_i[0,1])\subseteq \mathcal N_{\tilde P_i}(\tilde x)$ fulfills (\ref{ineq}). 
		
		Again, using the fact that $\tilde x_i^*\in co(\mathcal N_{\tilde P_i}(\tilde x)\cap S_i[0,1])$, one can obtain $x_{1i}^*,\cdots, x_{ri}^*\in \mathcal N_{\tilde P_i}(\tilde x)\cap S_i[0,1]$ with $\lambda_1,\cdots, \lambda_r\in [0,1]$ satisfying $\sum_{k=1}^{r}\lambda_k=1$ and,
		$$0= \tilde x_i^*=\sum_{k=1}^{r}\lambda_k x_{ki}^*.$$
		Suppose we have $k_\circ\in \{1,\cdots, r\}$ with $\lambda_{k_\circ} >0$. Then,
		$$ -x_{k_\circ i}^*=\sum_{k\neq k_\circ,k=1}^{r} \frac{\lambda_k}{\lambda_{k_\circ}} x^*_{ki}.$$
		We get $-x_{k_\circ i}^*\in \mathcal{N}_{\tilde P_i}(\tilde x)$ as it is a convex cone. Hence, $x_{k_\circ i}^*\in \mathcal{N}_{\tilde P_i}(\tilde x)\cap -\mathcal{N}_{\tilde P_i}(\tilde x)$. However, $\tilde P_i(\tilde x)$ is an open set due to our assumptions on $P_i$ and it also contains $z_i$. Therefore, $\mathcal{N}_{\tilde P_i}(\tilde x)\cap -\mathcal{N}_{\tilde P_i}(\tilde x)=\{0\}=\{x_{k_\circ i}^*\}$ by \cite[Lemma 2.1]{cotrina}. But, this contradicts the fact that $x_{k_\circ i}^*\in S_i[0,1]$. Consequently, our hypothesis is false and $\tilde x$ becomes an equilibrium for the game $\Gamma'$.
	\end{proof}
	Finally, we provide the sufficient conditions required for the occurrence of equilibrium for jointly convex GNEP $\Gamma'=(\mathcal{X}, K_i, P_i)_{i\in \Lambda}$ by employing the above characterization of $\Gamma'$ in terms of variational inequality.  
	\begin{theorem}\label{existence}
		Suppose $\mathcal X$ is a non-empty closed convex subset of $\mathbb{R}^n$ and $K_i:X\rightrightarrows X_i$ is defined as (\ref{constraint}). Then there exists a solution for $\Gamma'=(\mathcal{X}, K_i, P_i)_{i\in \Lambda}$ if,
		\begin{itemize}
			\item[(a)] for each $i\in \Lambda$, the map $P_i:X\rightrightarrows X_i$ is lower semi-continuous with open upper sections satisfying $x_i\notin co(P_i(x))$ for any $x\in \mathbb{R}^{n}$; 
			\item[(b)] there exists $\rho>0$ such that for each $x\in \mathcal{X}\setminus \bar B(0,\rho)$ a vector $z\in \mathcal{X}$ exists in such a way that $\norm{z}<\norm{x}$ and $z_i\in P_i(x)$ for each $i\in \Lambda$;
			\item[(c)] there exists $\rho'>\rho$ such that $\mathcal X\cap \bar B(0,\rho')\neq \emptyset$. 
		\end{itemize}  
	\end{theorem}
	\begin{proof}
	\textcolor{black}	{
		In order to prove this result, it is enough to show that $VI(T,\mathcal X)$ admits a solution as per Theorem \ref{VI}. 
		 In the view of Lemma \ref{usc}, we know that $T$ is an u.s.c. map with non-empty convex compact values. Further, the set $\mathcal{X}_{\rho'}=\mathcal{X}\cap \bar B(0,\rho')$ is non-empty convex compact. 
		 Therefore, by using \cite[Theorem 9.9]{aubin} we obtain $\tilde x\in \mathcal X_{\rho'}$ and $\tilde x^*\in T(\tilde x)$ which satisfy,}
		\begin{equation}\label{VIeq}
			\langle \tilde x^*, y-\tilde x \rangle \geq 0~\text{for each}~y\in \mathcal{X}_{\rho'}.
		\end{equation}
		
		We intend to show that inequality (\ref{VIeq}) indeed holds for each $z\in \mathcal{X}$, that is,
		\begin{equation}\label{VIunbound}
			\langle \tilde x^*, z-\tilde x \rangle \geq 0~\text{for each}~z\in \mathcal{X}.
		\end{equation}
		In this regard, one can observe that there exists $z^\circ \in \mathcal X\cap B(0,\rho')$ such that,
		\begin{equation}\label{VIeqa}
			\langle x^*,z^\circ-\tilde x\rangle\leq 0~\text{for each}~x^*\in T(\tilde x).
		\end{equation}
		In fact, one can consider $z^\circ=\tilde x$ if $\norm{\tilde x}<\rho$. Otherwise, if $\tilde x=\rho'>\rho$ then we obtain some $z^\circ\in\mathcal{X}$ with $\norm{z^\circ}<\rho'$ and $z^\circ_i\in P_i(\tilde x)$ due to our hypothesis. We know that $\mathcal N_{P_i}(x_{-i},x_i)=\{x_i^*|\,\langle x_i^*, y_i-x_i\rangle\leq 0,~\forall\,y_i\in P_i(x_{-i},x_i)\}$ as per (\ref{normal}). Therefore, the fact that $z^\circ_i\in P_i(\tilde x)$ for any $i\in \Lambda$ implies,
		\begin{equation}
			\langle x_i^*, z^\circ_i-\tilde x_i\rangle \leq 0,~\text{for all}~x_i^*\in \mathcal N_{P_i}(\tilde x_{-i}, \tilde x_i).
		\end{equation}
		Since $T_i(\tilde x)\subseteq \mathcal{N}_{P_i}(\tilde x)$, we observe that (\ref{VIeqa}) holds true.
		
		To prove (\ref{VIunbound}), suppose $z\in \mathcal{X}\setminus \bar B(0,\rho')$ is arbitrary. Since $\mathcal X$ is convex and $z_\circ \in \mathcal X\cap B(0,\rho')$, we obtain some $t\in (0,1)$ such that $tz+(1-t)z_\circ\in \mathcal{X}_{\rho'}$. Hence, in the view of (\ref{VIeq}) we have,
		\begin{equation}\label{VIeqb}
			\langle \tilde x^*, tz+(1-t)z_\circ-\tilde x \rangle \geq 0.
		\end{equation}
		We observe that inequality (\ref{VIunbound}) follows on combining inequality (\ref{VIeqb}) with (\ref{VIeqa}). Finally, the vector $\tilde x\in\mathcal{X}$ which solves $VI(T,\mathcal{X})$ (\ref{VIeqb}) is an equilibrium for the jointly convex GNEP $\Gamma'$ by virtue of Theorem \ref{VI}.
		
	\end{proof}
	\begin{remark}\label{boundedVI}
		\begin{itemize}
			\item[(i)] If $\mathcal{X}$ is bounded, the coercivity criterion in hypotheses (b) and (c) becomes true by assuming that $\rho>\sup \{x|\,x\in \mathcal{X}\}$.
			\item[(ii)] The coercivity criterion assumed by us in the hypotheses (b) and (c) is motivated from the coercivity criterion for variational inequalities with unbounded constraint sets \cite[Theorem 2.1]{aussel-hadj}. 
			\item[(iii)] The coercivity criterion assumed by us in the hypothesis (b) can be interpreted as follows: there exists $\rho>0$ such that for any vector $x=(x_{-i},x_i)$ with the norm larger than $\rho>0$ there exists a vector $z=(z_{-i},z_i)$ whose norm is less than the norm of vector $x$ and any player $i$ prefers $z_i$ over $x_i$. It is easy to notice that the coercivity condition in hypothesis (b) is satisfied if the players can always find the preferable strategies in some bounded subset of $\mathcal{X}$. 
		\end{itemize}
	\end{remark}
	\subsection{GNEP with Preference maps and Quasi-variational Inequalities}
	In this section, we state the suitable conditions under which any solution of a QVI problem becomes a generalized Nash equilibrium for the GNEP having non-ordered non-convex inter-dependent preferences \cite{shafer,tian,yannelis}. Moreover, we investigate the sufficient conditions required for solvability of such games. 
	
	First of all, let us demonstrate that any solution of a quasi-variational inequality is an equilibrium for the considered generalized Nash game $\Gamma$  under suitable conditions.
	\begin{theorem}\label{equivalence}
		Assume that the following assumptions hold for any $i\in \Lambda$:
		\begin{itemize}
			\item [(a)] $K_i:X\rightrightarrows X_i$ admits non-empty values;
			\item[(b)] $P_i:X\rightrightarrows X_i$ has open upper section.
		\end{itemize}
		Suppose the map $T$ is defined as (\ref{T}) and the map $K:X\rightrightarrows X$ is defined as $K(x)=\prod_{i\in \Lambda} K_i(x)$. Then, any solution of $QVI(T,K)$ is an equilibrium for $\Gamma =(X_i,K_i,P_i)_{i\in \Lambda}$. 
	\end{theorem}
	\begin{proof}
		Suppose $\tilde x$ solves $QVI(T,K)$ then $\tilde x\in K(\tilde x)$ satisfies,
		\begin{equation}\label{GNEPQVI}
			\exists~\tilde x^*\in T(\tilde x), \langle \tilde x^*, y-\tilde x\rangle \geq 0~\text{for all}~ y\in K(\tilde x).
		\end{equation}
		
		Let us assume that $\tilde x$ is not an equilibrium for the given generalized Nash game $\Gamma$. Then, $co(P_i(\tilde x))\cap  K_i(\tilde x)\neq \emptyset$ for some $i\in \Lambda$. Suppose $z_i\in co(P_i(\tilde x))\cap  K_i(\tilde x)$. Then, clearly $y=(\tilde x_{-i}, z_i)\in K(\tilde x)$ and from (\ref{GNEPQVI}) we obtain,
		\begin{equation}\label{rel}
			\langle \tilde x_i^*, z_i-\tilde x_i\rangle \geq 0.
		\end{equation}
		By virtue of Proposition \ref{relation}, one can observe that $\tilde x_i^*=0$ because $\tilde x_i^*\in T_i(\tilde x) \subseteq \mathcal N_{\tilde P_i}(\tilde x)$ fulfills (\ref{rel}). 
		
		Since $0=\tilde x_i^*\in co(\mathcal N_{\tilde P_i}(\tilde x)\cap S_i[0,1])$, one can obtain $x_{1i}^*,\cdots, x_{ri}^*\in \mathcal N_{\tilde P_i}(\tilde x)\cap S_i[0,1]$ with $\lambda_1,\cdots, \lambda_r\in [0,1]$ satisfying $\sum_{k=1}^{r}\lambda_k=1$ and $0= \tilde x_i^*=\sum_{k=1}^{r}\lambda_k x_{ki}^*.$ By following the argument presented in the proof of Theorem \ref{VI}, one can easily obtain a contradiction to the fact that $x_{ki}^*\in S_i[0,1]$ for each $ki$.
		
	\end{proof}
	
	One can notice the practical importance of GNEP with non-ordered inter-dependent preferences and unbounded constraint maps by referring to \cite{yannelis,yannelis2022}, where the case of Walrasian equilibrium problems is recently considered. In an existing work on GNEP with non-ordered inter-dependent preferences by Tian \cite{tian}, the GNEP with unbounded strategy sets but compact valued strategy maps is studied. In particular, Tian \cite[Theorem 2]{tian} employed a coercivity condition to ensure the occurrence of equilibrium for such GNEP by using a fixed point result.
	
	Our aim is to derive the occurrence of equilibrium for GNEP in which the players have non-ordered inter-dependent preferences, unbounded strategy sets and non-compact valued strategy maps by using variational techniques. For this purpose, we will employ the following coercivity criterion: Suppose $P_i:X\rightrightarrows X_i$ and $K_i:X\rightrightarrows X_i$ denote the preference and strategy maps of player $i$, respectively. Let $K:X\rightrightarrows X$ be formed as $K(x)=\prod_{i\in \Lambda} K_i(x)$. Then, the coercivity criterion ($\mathcal{C}_x$) holds at any $x\in X$ if,
	\begin{align*}
		(\mathcal{C}_x): \enspace \exists\,\rho_x & >0~\text{such that}~\forall\,y\in K(x)\setminus\bar B(0,\rho_x),\exists\, z\in K(x)\\& \quad \text{with}~\norm{z} <\norm{y}~\text{and}~z_i\in P_i(y)~\forall\,i\in\Lambda.
	\end{align*}
	\begin{remark}
		\begin{itemize}
			\item[(i)] For any $x\in X$, if $K_i(x)$ is bounded for each $i\in \Lambda$ then the coercivity criterion ($\mathcal{C}_x$) is fulfilled by assuming that $\rho_x> \sup \{\norm{y}\,|\,y\in K(x)\}$.
			\item[(ii)] One can observe that the coercivity criterion ($\mathcal{C}_x$) is comparable with the existing coercivity criterion for GNEP having numerical representation of preferences (as described in Remark \ref{numGNEP}). In fact, on strengthening the \textit{coerciveness condition} (1) in \cite{cotrinatime} to following form: for any $x\in X$ there exists a non-empty convex compact subset $H_x$ such that,
			\begin{align}
				\text{for any}~y\in K(x)\setminus & H_x~\text{there exists}~ z\in K(x)\cap H_x~\text{satisfying,}\notag\\&u_i(y_{-i},z_i)> u_i(y_{-i},y_i);\label{numGNEP1}
			\end{align}
			we observe that ($\mathcal{C}_x$) is fulfilled if (\ref{numGNEP1}) holds. One can verify this by simply assuming $\rho_x>\sup\, \{y|\,y\in H_x\}$.   
		\end{itemize}
	\end{remark}
	
	Finally, we check the solvability of the generalized Nash equilibrium problem $\Gamma= (X_i,K_i,P_i)_{i\in\Lambda}$ by employing the characterization of $\Gamma$ in terms of quasi-variational inequality derived in Theorem \ref{equivalence}.
	\begin{theorem} \label{QVImainresult}
		Suppose that the following assumptions hold for any $i\in \Lambda$:
		\begin{itemize}
			\item[(a)] $X_i$ is a non-empty, closed and convex subset of $\mathbb{R}^{n_i}$;
			\item[(b)] $K_i$ is closed and lower semi-continuous map with $K_i(x)$ being non-empty convex for any $x\in X$;
			\item[(c)] $P_i$ is lower semi-continuous map with open upper sections and $x_i\notin co(P_i(x))$ for any $x\in X$.
		\end{itemize}
		Then, there exists an equilibrium for $\Gamma =(X_i,K_i,P_i)_{i\in\Lambda}$ if the coercivity criterion ($\mathcal{C}_x$) holds at any $x\in X$ and there exists $\rho>\sup\{\rho_x\,|\,x\in X\}$ in such a way that $(\prod_{i\in \Lambda} K_i(x))\cap \bar B(0,\rho)\neq \emptyset$ for each $x\in X$.
	\end{theorem}
	\begin{proof}
		In order to prove this result, it is enough to show that $QVI(T,K)$ admits a solution according to Theorem \ref{equivalence}. Suppose $X_\rho =X\cap \bar B(0,\rho)$ and the map $K_\rho:X_\rho\rightrightarrows X_\rho$ is formed as $K_\rho (x)= K(x)\cap \bar B(0,\rho)$ where $K(x)=\prod_{i\in \Lambda} K_i(x)$. 
		We claim that $K_\rho$ is a lower semi-continuous map. 
		In fact, one can observe that $K(x)\cap \bar B(0,\rho)\neq \emptyset$ as per our hypothesis. Hence, there exists some $z\in K(x)\cap B(0,\rho)$ due to coercivity criterion ($\mathcal{C}_x$). Then, we observe that $K_\rho$ is lower semi-continuous map as per \cite[Lemma 2.3]{cotrinaeq}. Consequently, $K_\rho$ is closed and lower semi-continuous map and $K_\rho (x)$ is non-empty convex compact for any $x\in X$.
		
		According to Lemma \ref{usc}, $T:X_\rho \rightrightarrows \mathbb{R}^n$ is a u.s.c. map with non-empty convex and compact values. As per \cite[Corollary(Theorem 3)]{tan}, we obtain some $\tilde x\in K_\rho (\tilde x)$ such that,
		\begin{equation}\label{coerc}
			\text{there exists}~\tilde x^*\in T(\tilde x),\langle\tilde x^*,y-\tilde x\rangle \geq 0,~\text{for all}~y\in K_\rho(\tilde x).
		\end{equation}
		We claim that $\tilde x$ is a solution of $QVI(T,K)$, that is,
		\begin{equation}\label{coere}
			\text{there exists}~\tilde x^*\in T(\tilde x),\langle\tilde x^*,y-\tilde x\rangle \geq 0,~\text{for all}~y\in K(\tilde x).
		\end{equation}
		
		In this regard, we first show that there exists $z\in K(x)\cap B(0,\rho)$ such that,
		\begin{equation}\label{coera}
			\langle x^*,z-\tilde x\rangle \leq 0,~\text{for all}~x^*\in T(\tilde x).
		\end{equation}
		Since $\tilde x\in K_\rho(\tilde x)$, one can observe that (\ref{coera}) holds true if $\norm{\tilde x}<\rho$. On the other hand, in the case $\norm{\tilde x}=\rho$ then $\tilde x\in K(\tilde x)\setminus \bar B(0,\rho_{\tilde x})$. Hence, there exists $z\in K(\tilde x)$ with $\norm z<\rho$ and $z_i\in P_i(\tilde x)$ for all $i\in \Lambda$. Then, for any $i\in \Lambda$ we obtain, 
		\begin{equation}\label{coerb}
			\langle x_i^*, z_i-\tilde x_i\rangle \leq 0,~\text{for all}~x_i^*\in \mathcal N_{P_i}(\tilde x_{-i},\tilde x_i)
		\end{equation}
		by definition of $\mathcal N_{P_i}$ (see (\ref{normal})). Since $T_i(\tilde x)\subseteq \mathcal N_{P_i}(\tilde x)$, we observe that (\ref{coera}) holds true. Finally, for any $y\in K(\tilde x)\setminus\bar B(0,\rho)$ we obtain some $t\in(0,1)$ such that $ty+(1-t)z\in K(\tilde x)\cap \bar B(0,r)$ by using convexity of $K(\tilde x)$. Hence, in the view of (\ref{coerc}) we observe,
		\begin{equation}\label{coerd}
			\langle\tilde x^*,ty+(1-t)z-\tilde x \rangle\geq 0.
		\end{equation}
		We observe that inequality (\ref{coere}) follows on combing (\ref{coerd}) with (\ref{coera}). Finally, the vector $\tilde x$ which solves $QVI(T,K)$ (\ref{coere}) is an equilibrium for the given GNEP $\Gamma$ by virtue of Theorem \ref{equivalence}. 
	\end{proof}
\textcolor{black}{We observe that the coercivity condition $(\mathcal{C}_x)$ holds for any $x\in X$ if $X_i$ is bounded for any $i\in \Lambda$. Hence, we obtain following result as a corollary of Theorem \ref{QVImainresult}.
\begin{corollary}\label{QVIcoro}
	Suppose that the following assumptions hold for any $i\in \Lambda$:
	\begin{itemize}
		\item[(a)] $X_i$ is a non-empty, compact and convex subset of $\mathbb{R}^{n_i}$;
		\item[(b)] $K_i$ is closed and lower semi-continuous map with $K_i(x)$ being non-empty convex for any $x\in X$;
		\item[(c)] $P_i$ is lower semi-continuous map with open upper sections and $x_i\notin co(P_i(x))$ for any $x\in X$.
	\end{itemize}
	Then, there exists an equilibrium for $\Gamma =(X_i,K_i,P_i)_{i\in\Lambda}$.
\end{corollary}}
	\begin{remark}
		\begin{itemize}
			\item[(i)] One can observe that \cite[Theorem 1]{shafer} and \cite[Theorem 1]{tian} follows from above-stated result Corollary \ref{QVIcoro}. 
			\item[(ii)] The coercivity criterion assumed by us in the above theorem is motivated from the coercivity criterion for quasi-variational inequalities with unbounded constraint maps in \cite{ausselcoer}.
		\end{itemize}
	\end{remark}
	
	\section{GNEP with Preference Maps in terms of Binary Relations}\label{secbinary}
	
	In this section, we establish the existence of equilibrium for the GNEP in which the preferences of consumers are expressed in terms of binary relations. 
	Let us assume that the binary relation `$\succ_i$' characterizes the preferences of $i^{th}$ player over the set $X$. For any given strategy vector $x_{-i}\in \prod_{j\in (\Lambda\setminus\{i\})} X_j$ of rival players, whenever the consumer strictly prefers some strategy $y_i\in X_i$ over $x_i$, we represent this situation as $(x_{-i},y_i)\succ_i (x_{-i},x_i)$ \cite{bade}. Clearly, the preference map $P_i:X\rightrightarrows X_i$ can be defined in terms of binary relation as,
	\begin{equation}\label{pref}
		P_i(x_{-i},x_i)=\{y_i\in X_i\,|\, (x_{-i},y_i)\succ_i (x_{-i},x_i)\}.
	\end{equation} 
	Moreover, the definition discussed in Section 3 can be given in terms of binary relations as follows: Suppose the $N$-ordered triples $\Gamma=(X_i,K_i,\succ_i)_{i\in \Lambda}$ represents the generalized Nash game defined by preference relations. A vector $\tilde x\in X$ is known as equilibrium for the game $\Gamma$ if $\tilde x\in K_i(\tilde x)$ and for each $i\in \Lambda$ there is no such action $y_i\in K_i(\tilde x)$ such that $(\tilde x_{-i},y_i)\succ_i (\tilde x_{-i},\tilde x_i)$ \cite{bade}.
	
	Let us recollect some important properties of preference relation from \cite{kreps}.
	\begin{definition}
		Suppose $\succ_i$ denotes the preference relation over the set $X$ \cite{debreubook,krepsweak,milasipref2021}. Then $\succ_i$ is known as,
		\begin{itemize}
			\item[(a)] irreflexive if there is no such $x\in X$ for which we have $x\succ_i x$;
			\item[(b)] asymmetric if there is no $x,y\in X$ such that $x\succ_i y$ and $y\succ_i x$;
			\item[(c)] negatively transitive if for any $x,y,z \in X$ we have $x\succ_i y$, then either $x\succ_i z$ or $z\succ_i y$, or both.
		\end{itemize}
	\end{definition}
	It is well known that strict preference relation $\succ_i$ induces a weak preference relation `$\succeq_i$' \cite{milasipref2021, kreps}. For any $x,y\in X$ we have $x\succeq_i y$, that is, player $i$ weakly prefers $x$ over $y$ if $y\succ_i x$ is not the case. Furthermore, if the  strict preference relation $\succ_i$ is asymmetric and negatively transitive then the corresponding weak preference relation, 
	\begin{itemize}
		\item[(a)] $\succeq_i$ is complete: For any $x,y\in X$ either $x\succeq_i y$ or $y\succeq_i x$;
		\item[(b)] $\succeq_i$ is transitive: For any $x,y,z\in X$ if $x\succeq_i z$ and $z\succeq_i y$ then $x\succeq_i y$,
	\end{itemize}
	respectively. We say that the preference relation $\succeq_i$ (or $\succ_i$) is non-ordered if it is not complete and transitive (it is not asymmetric and negatively transitive). 
	Let us discuss the continuity and convexity properties of preference relations which play an important role in establishing the existence of equilibrium for the proposed games. 
	\begin{itemize}
		\item[(a)] $\succ_i$ is said to be l.s.c. if the set $\{x\in X\,|\, x\succ_i y\}$ is open for any $y\in X$ and it is said to be u.s.c. if the set $\{x\in X\,|\, y\succ_i x\}$ is open for any $y\in X$;
		\item[(b)] $\succ_i$ is said to be continuous if it is both u.s.c. and l.s.c. Alternatively, `$\succ_i$' is continuous iff the induced weak preference relation `$\succeq_i$' satisfies $x_n\succeq_i y_n,~\forall n\implies x\succeq_i y$ for any sequences $\{x_n\}$ and $\{y_n\}$ converging to $x$ and $y$ respectively;
		\item[(c)] $\succ_i$ is convex if for any $x,y,z\in X$ satisfying $x\succ_i z$ and $y\succ_i z$ we have $tx+(1-t)y\succ_i z$ for each $t\in (0,1)$;
		\item[(d)] $\succ_i$ is said to be non-satiated if for any $x\in X$ there exists $z\in X$ satisfying $z\succ_i x$.
	\end{itemize}
	Following result will be useful in proving the upcoming existence result on generalized Nash equilibrium problems with preference maps in terms of binary relations:
	\begin{proposition} \label{binary}
		Assume that $\succ_i$ is the strict preference relation defined over $X$ for player $i$. Suppose the map $P_i:X\rightrightarrows X_i$ be defined as (\ref{pref}). Then,
		\begin{itemize}
			\item[(a)] $x_i\notin P_i(x)$ for any $x\in X$ if $\succ_i$ is irreflexive;
			\item[(b)] $P_i(x)\neq \emptyset$ for any $x\in X$ if $\succ_i$ is non-satiated on $X$;
			\item[(c)] $P_i$ admits open upper sections if $\succ_i$ is lower semi-continuous;
			\item[(d)] $P_i$ is lower semi-continuous if $\succ_i$ is continuous.
		\end{itemize}
	\end{proposition}
	\begin{proof}
		\begin{itemize}
			\item[(a)] Since $\succ_i$ is irreflexive, we can not have $(x_{-i},x_i)\succ_i(x_{-i},x_i)$ for any $x\in X$. Hence, $x_i\notin P_i(x)$ for any $x\in X$.
			\item[(b)] It follows from definition directly.
			\item[(c)] Suppose $x\in X$ is arbitrary. If $P_i(x)=\emptyset$ then the case follows trivially. Suppose $y_i\in P_i(x)$ then we claim that for any sequence $\{(y_i)_n\}$ converging to $y_i$ there exists some $n_\circ \in \mathbb{N}$ such that $(y_i)_n\in P_i(x)$ for all $n\geq n_\circ$. In fact, $y_i\in P_i(x)$ implies $(x_{-i},y_i)\succ_i(x_{-i},x_i)$. Hence, on taking $y=(x_{-i},y_i)$ and $y_n=(x_{-i},(y_i)_n)$ we observe that the sequence $\{y_n\}$ converges to $y$. Now, by using lower semi-continuity of $\succ_i$ we get some $n_\circ\in \mathbb{N}$ such that $y_n\succ_i x$ for all $n\geq n_\circ$. Hence, $(y_i)_n\in P_i(x)$ for all $n\geq n_\circ$.
			\item[(d)] Suppose $P_i$ is not lower semi-continuous at some $x=(x_{-i},x_i)\in X$ then there exists $V_i\subset X_i$ with $P_i(x_{-i},x_i)\cap V_i\neq \emptyset$ and for each $n\in \mathbb{N}$ there exists $x_n\in B(x,\frac{1}{n})$ such that $P_i(x_n)\cap V_i=\emptyset$ for all $n\in \mathbb{N}$. Consequently,
			\begin{equation} \label{cont}
				((x_{-i})_n,(x_i)_n) \succeq_i ((x_{-i})_n,z_i)~\forall\,n\in \mathbb{N},~\forall\,z_i\in V_i.
			\end{equation}
			Suppose $z_i\in V_i$ is arbitrary. Assume that $z_n=((x_{-i})_n, z_i)$ and $z=(x_{-i},z_i)$. Clearly, the sequence $\{z_n\}$ converges to $z$. Now, in the view of (\ref{cont}) we get $x\succeq_i z$ by using the continuity of $\succ_i$. Since $z_i$ was chosen arbitrarily, we observe $(x_{-i},x_i) \succeq_i (x_{-i},z_i) \text{ holds for all } z_i\in V_i.$
			But, this contradicts $P_i(x_{-i},x_i)\cap V_i\neq \emptyset$. Finally, we affirm that $P_i$ is l.s.c.
		\end{itemize}
	\end{proof}
	Following result ensures the occurrence of equilibrium for the jointly convex GNEP in which the preference relations of players are incomplete and non-transitive. It is obtained by combining Theorem \ref{existence}, Remark \ref{boundedVI}(i) and Proposition \ref{binary}.
	\begin{theorem}
		Assume that $\mathcal X$ is a convex, compact and non-empty subset of $\mathbb{R}^n$. Suppose for any $i\in \Lambda$, the binary relation $\succ_i$ is irreflexive, convex and continuous over the set $X$. Suppose $K_i:X\rightrightarrows X_i$ is defined as (\ref{constraint}). Then, there exists a solution for $\Gamma'=(\mathcal X,K_i,\succ_i)_{i\in }$.
	\end{theorem}
	Furthermore, the following result proves the occurrence of equilibrium for generalized Nash equilibrium problems in which the preference relations of players are incomplete and non-transitive. It is obtained by combining Corollary \ref{QVIcoro} and Proposition \ref{binary}.
	\begin{theorem}
		Suppose that the following assumptions hold for any $i\in \Lambda$:
		\begin{itemize}
			\item[(a)] $X_i$ is a non-empty, convex and compact subset of $\mathbb{R}^{n_i}$;
			\item[(b)] $K_i:X\rightrightarrows X_i$ is closed and lower semi-continuous map with $K_i(x)$ being non-empty convex set for any $x\in X$;
			\item[(c)] $\succ_i$ is irreflexive, convex and continuous over the set $X$.
		\end{itemize}
		Then, there exists a solution for $\Gamma=(X_i,K_i,\succ_i)_{i\in \Lambda}$.
	\end{theorem}
	
	\section{Applications}
	In this section, we apply the theoretical results established by us to demonstrate the existence of equilibrium for Arrow-Debreu market economy under uncertainty \cite{debreubook}. 
	These problems are recently studied in \cite{milasipref, milasipref2021} by using variational approach. It is worth noticing that Milasi et. al. in \cite{milasipref, milasipref2021} considered the economic equilibrium problems under uncertainty in which the preferences of agents are convex (or semi-strictly convex), incomplete and independent of the strategies chosen by other agents. We check the solvability of economic equilibrium problem under uncertainty in which the preferences of agents are non-convex, non-ordered, price-dependent and inter-dependent.
	
	Consider an economy consisting of finite number of agents ($I$ consumers and $J$ producers) actively trading $L$ commodities. Suppose $\mathcal I=\{1,2,\cdots I\}$, $\mathcal J=\{1,2,\cdots J\}$ and $\mathcal{L}=\{1,2,\cdots L\}$ denote the sets of consumers, producers and commodities, respectively. Assume that these agents trade in two periods of time, $t=0$ (today) and $t=1$ (tomorrow). The market at $t=0$ is generally represented by the state $\mathcal S_0=\{0\}$. Further, the uncertainty at $t=1$ is represented in the form of states which may take place tomorrow. Suppose the set $\mathcal S_1=\{1,2,\cdots s_1\}$ consists of all such states which possibly occur at $t=1$. 
	Then, the set $\mathcal S=\mathcal S_0\cup \mathcal S_1$ consists of $S=s_1+1$ states. We assume that $\mathcal E=\{\zeta_0,\zeta_1,\cdots \zeta_{s_1}\}$ represents the set of all possible situations occurring in these states. 
	
	In this economy with uncertainty, we assume that the evolution of market is indicated by using an oriented graph $\mathcal G$ with nodes $\mathcal E$ and root $\zeta_0$. Every node $\zeta_s\in\mathcal{E}$ in this graph indicates the \textit{contingency} plan of market in the state $s$. A state-contingent commodity $x_i^l(\zeta_s)\geq 0$ 
	indicates the quantity of physical commodity $l\in \mathcal{L}$ received by consumer $i\in \mathcal{I}$ if the state $s\in \mathcal{S}$ occurs. Thus, $a_i=((a_i^l(\zeta_s))_{l\in \mathcal{L}})_{s\in \mathcal{S}}$ represents the state-contingent commodity vector of consumer $i$ in the consumption set $A_i\subseteq \mathbb{R}_+^{H}$ where $H=L(s_1+1)$. We indicate $b_j^l(\zeta_s)\geq 0$ (or $\leq 0$) as output (or input) of the commodity $l\in\mathcal{L}$ by a production unit $j\in J$ if the state $s\in \mathcal{S}$ occurs. Thus, a vector $b_j=((b_j^l(\zeta_s))_{l\in \mathcal{L}})_{s\in \mathcal{S}}$ denotes the state-contingent commodity vector of producer $j$ in the set of possible production plans $B_j\subseteq \mathbb{R}^{H}$. Moreover, $p=((p^l(\zeta_s))_{l\in \mathcal{L}})_{s\in \mathcal{S}}\in \mathbb{R}^{H}$ is the system of prices for the state-contingent commodities. 
	
	Consider that the state-contingent commodity vector $e_i=((e_i^l(\zeta_s))_{l\in \mathcal{L}})_{s\in \mathcal{S}}$ is an initial endowment for any consumer $i$, which satisfies $e_i^l(\zeta_s)> \hat x_i^l(\zeta_s)$ for some fixed $\hat x_i=((\hat x_i^l(\zeta_s))_{l\in \mathcal{L}})_{s\in \mathcal{S}}\in X_i$. In the total production, the share of any consumer $i$ is $\sum_{j\in \mathcal{J}}\theta_{ij} y_j$, where the fixed weights $\theta_{ij}$ satisfy $\sum_{i\in \mathcal {I}}\theta_{ij}=1$. Suppose that the price vectors of contingent commodities, that belong to the simplex $\mathbb{R}_+^{H}$, are normalized so that the price set becomes,
	$$\Delta=\bigg\{p\in\mathbb{R}_+^{H}|\, \sum_{l\in \mathcal{L},s\in \mathcal{S}}p^l(\zeta_s)=1\bigg\}.$$
	For a given price vector $p$ and production vector $b=(b_j)_{j\in \mathcal{J}}$ of contingent commodities, the total wealth of consumer $i$ is $\langle p, e_i +\sum_{j\in \mathcal{J}}\theta_{ij} y_j\rangle$. Since the expenditure of consumer on contingent commodities can not exceed his total existing wealth, any consumer $i$ has the following budget constraint,
	$$M_i(p,b)=\bigg\{a_i\in A_i|\,\langle p,a_i\rangle\leq \langle p,e_i\rangle +\max \bigg[0,\sum_{j\in \mathcal{J}}\theta_{ij}\langle p, b_j\rangle\bigg]\bigg \}.$$
	Let us denote,
	$$A=\prod_{i\in \mathcal{I}} A_i~\text{and}~ B= \prod_{j\in \mathcal{J}} B_j.$$ 
	Motivated by the Walrasian equilibrium problem considered in \cite{shafer, tian, yannelis}, we assume that every consumer $i\in \mathcal{I}$ has a preference map $\hat P_i:A\times B\times\Delta\rightrightarrows A_i$ that depends on the choice of other agents and price of contingent commodities, simultaneously. Further, we define the competitive equilibrium for Arrow-Debreu market economy under uncertainty as follows which is motivated by \cite{faccsurvey,yannelis,milasipref,cotrinatime}:
	\begin{definition} \label{defequili}
		A vector $(\tilde a_1,\tilde a_2,\cdots,\tilde a_I,\tilde b_1,\tilde b_2,\cdots \tilde b_J)\in A\times B\subset \mathbb{R}^{H(I+J)}$ and a price system $\tilde p=((\tilde p^l(\zeta_s))_{l\in \mathcal{L}})_{s\in\mathcal{S}}\in \Delta$ together forms a competitive equilibrium for Arrow-Debreu market economy under uncertainty if following conditions hold:
		\begin{align}
			\langle \tilde p,\tilde b_j\rangle =\max_{b_j\in B_j} \langle&\tilde p,b_j\rangle,~\text{for each}~j\in \mathcal{J};\\
			\hat P_i(\tilde a,\tilde b,\tilde p)\cap M_i(\tilde p,&\tilde b)= \emptyset, ~\text{for each}~ i\in \mathcal{I};\label{res2a}\\
			\bigg \langle \tilde p, \sum_{i\in \mathcal{I}} (\tilde a_i- e_i) -\sum_{j\in \mathcal{J}} \tilde b_j\bigg \rangle &=\max_{p\in \Delta}\, \bigg\langle p, \sum_{i\in \mathcal{I}} (\tilde a_i- e_i)-\sum_{j\in \mathcal{J}} \tilde b_j\bigg\rangle.	\label{req1}
		\end{align} 
	\end{definition}
	Following lemma states that for any contingent commodity, the number of units available for consumption cannot exceed the sum of initial endowments and the units produced. In particular, it shows that the Definition \ref{defequili} implies the \cite[Definition 3.2]{milasipref}.
	\begin{lemma}\label{competitive}
		Assume that the set $B_j\subseteq \mathbb{R}^H$ contains $0$ vector for each $j\in \mathcal{J}$. Then, we have:
		\begin{itemize}
			\item[(a)] $(\tilde a, \tilde b,\tilde p)$ fulfills,
			\begin{equation}\label{res1}
				\sum_{i\in \mathcal{I}} \tilde a_i^l(\zeta_s)\leq \sum_{j\in \mathcal{J}} \tilde b_j^l(\zeta_s) +\sum_{i\in \mathcal{I}} e_i^l(\zeta_s),~\text{for any}~l\in \mathcal{L}~\text{and}~s\in \mathcal{S};
			\end{equation} 
			\item[(b)] $(\tilde a, \tilde b,\tilde p)$ fulfills, 
			\begin{equation}\label{res2}
				\langle \tilde p, \sum_{i\in \mathcal{I}} (\tilde a_i- e_i)-\sum_{j\in \mathcal{J}} \tilde b_j\rangle=0
			\end{equation}
			if for any $i\in \mathcal{I}$, $\hat P_i$ admits non-empty convex values and open upper sections with $a_i\in cl(\hat P_i(a,b,p))$ for each $(a,b,p)\in  A\times B\times \Delta$.
		\end{itemize}
	\end{lemma}
	\begin{proof}
		\begin{itemize}
			\item[(a)] Since $0\in B_j$, we have,$\langle \tilde p,\tilde b_j\rangle \geq 0$ for any $j\in\mathcal{J}$. Further, $\tilde a_i\in M_i(\tilde p,\tilde b)$ for each $i\in \mathcal{I}$ implies,
			\begin{equation}\label{req2a}
				\langle \tilde p,\tilde a_i\rangle\leq \langle \tilde p,e_i\rangle +\sum_{j\in \mathcal{J}}\theta_{ij}\langle \tilde p, \tilde b_j\rangle~\text{for any}~i\in\mathcal{I}
			\end{equation}
			Taking sum over $i\in\mathcal{I}$ yields,
			\begin{equation} \label{req2}
				\sum_{i\in\mathcal I} \langle \tilde p, \tilde a_i-e_i\rangle -\sum_{j\in\mathcal{J}} \langle\tilde p, \tilde b_j\rangle \leq 0.
			\end{equation}
			Now, inequality (\ref{req2}) together with Definition \ref{defequili} equation (\ref{req1}) yields,
			\begin{equation}\label{req3}
				\sum_{i\in\mathcal I} \langle p, \tilde a_i-e_i\rangle -\sum_{j\in\mathcal{J}} \langle p, \tilde b_j\rangle \leq 0~\text{for any}~p\in\Delta.
			\end{equation}
			For some $l_\circ\in\mathcal{L}$ and $s_\circ\in\mathcal{S}$, suppose $p=(p^l(\zeta_s))$ is defined as,
			\begin{equation*}
				p^l(\zeta_s)=\begin{cases}
					1,~l=l_\circ,s=s_\circ,\\
					0,~\text{otherwise}.	
				\end{cases}
			\end{equation*}
			Then, clearly $p\in \Delta$ and by substituting $p$ in the inequality (\ref{req3}) we obtain,
			$$\sum_{i\in \mathcal{I}} \tilde a_i^{l_\circ}(\zeta_{s_\circ})\leq \sum_{j\in \mathcal{J}} \tilde b_j^{l_\circ}(\zeta_{s_\circ}) +\sum_{i\in \mathcal{I}} e_i^{l_\circ}(\zeta_{s_\circ}).$$
			Since $l_\circ$ and $s_\circ$ are chosen arbitrarily, we observe inequality (\ref{res1}) follows.
			\item[(b)] We claim that 
			\begin{equation}\label{res2c}
				\langle\tilde p, \tilde a_i\rangle=\langle\tilde p, e_i\rangle+ \sum_{j\in \mathcal{J}}\theta_{ij}\langle p, \tilde b_j\rangle
			\end{equation} 
			for each $i\in \mathcal{I}$. In fact, in the view of inequality (\ref{req2a}), if we suppose that 
			\begin{equation}\label{res2b}
				\langle\tilde p, \tilde a_i\rangle<\langle\tilde p, e_i\rangle+ \sum_{j\in \mathcal{J}}\theta_{ij}\langle p, \tilde b_j\rangle
			\end{equation} 
			for some $i\in \mathcal{I}$. Then, we know that $\hat P_i(\tilde a,\tilde b,\tilde p)$ is non-empty convex open set with $\tilde a_i\in cl(\hat P_i(\tilde a,\tilde b,\tilde p))$. Hence, we obtain a sequence $\{a_i^n\}_{n\in\mathbb{N}}\subseteq \hat P_i(\tilde a,\tilde b,\tilde p)$ converging to $\tilde a_i$. According to (\ref{res2a}), $\hat P_i(\tilde a,\tilde b,\tilde p)\cap M_i(\tilde p,\tilde b)= \emptyset$ and we observe $a_i^n\notin M_i(\tilde p,\tilde b)$ for any $n\in\mathbb{N}$. Thus, we have 
			$$\langle\tilde p, \tilde a_i^n\rangle\geq \langle\tilde p, e_i\rangle+ \sum_{j\in \mathcal{J}}\theta_{ij}\langle p, \tilde b_j\rangle~\text{for any}~n\in \mathbb{N}.$$
			But, this contradicts our assumption in (\ref{res2b}) as we know that the sequence $\{a_i^n\}_{n\in\mathbb{N}}$ converges to $\tilde a_i$. Hence, our claim in (\ref{res2c}) holds true. Finally, we see equation (\ref{res2}) follows if we take sum over $i\in \mathcal{I}$ in equation (\ref{res2c}).
		\end{itemize}
	\end{proof}
	Now, we show the occurrence of the competitive equilibrium for Arrow-Debreu economy under uncertainty by using Theorem \ref{QVImainresult}.
	\begin{theorem}\label{comperesult}
		Assume that for any $i\in \mathcal{I}$ and $j\in \mathcal{J}$ we have,
		\begin{itemize}
			\item[(a)] the sets $A_i\subseteq \mathbb{R}_+^{H}$ and $B_j\subseteq \mathbb{R}^{H}$ are non-empty convex compact;
			\item[(b)] $\hat P_i$ is a lower semi-continuous map with non-empty convex values;
			\item[(c)] $\hat P_i$ admits open upper sections with $a_i\notin \hat P_i(a,b,p)$ but $a_i\in cl(\hat P_i(a,b,p))$ for each 
			$(a,b,p)\in A\times B\times \Delta$.
		\end{itemize}
		Then, the considered Arrow-Debreu economy under uncertainty admits atleast one competitive equilibrium.
	\end{theorem}
	\begin{proof}
		We observe that the given Arrow-Debreu economy under uncertainty is a specific instance of the generalized Nash equilibrium problem $\Gamma=(X_i,K_i,P_i)_{i\in \Lambda}$ considered in Section \ref{mainresult}. In fact, one can consider that there are $I+J+1$ players, including $I$ consumers, $J$ producers and one fictitious player \cite{yannelis}. These players regulates strategy vectors in the sets $X_i$, which are precisely given as,
		\begin{equation*}
			X_i=\begin{cases}
				A_i,&~i\in \{1,2,\cdots I\}\\
				B_{i-I},&~ i\in \{I+1, I+2,\cdots I+J\}\\
				P,&~i= I+J+1. 
			\end{cases}
		\end{equation*}
		Then, $x=(a,b,p)=(a_1,\cdots,a_I,b_1,\cdots,b_J,p)$ in the product set $X$ is given as:
		\begin{equation*}
			x_i=\begin{cases}
				a_i,&~i\in\{1,2,\cdots I\}\\
				b_{i-I},&~ i\in \{I+1, I+2,\cdots I+J\}\\
				p,&~i= I+J+1.
			\end{cases}
		\end{equation*}
		Further, the preference maps $P_i: X \rightrightarrows X_i$ for ${N}=I+J+1$ players in the set $\Lambda=\{1,\cdots,I,1,\cdots,J,I+J+1\}$ are formed as: For $I$ consumers we set $P_i=\hat P_i$, for $J$ producers we form $P_j:A\times B\times \Delta\rightrightarrows B_j$ as, $$P_{j}(a,b,p)=\big\{\bar b_j\in B_j|\,\langle p,\bar b_j\rangle>\langle p,b_j\rangle\big\}$$ and for the fictitious player we form $P_{I+J+1}:A\times B\times \Delta\rightrightarrows \Delta$ as, 
		$$P_{I+J+1}(a,b,p)=\bigg\{ \bar p\in \Delta~\bigg|\,\bigg\langle \bar p, \sum_{i\in \mathcal{I}} (a_i- e_i) -\sum_{j\in \mathcal{J}} b_j\bigg\rangle > \bigg\langle p, \sum_{i\in \mathcal{I}} (a_i- e_i) -\sum_{j\in \mathcal{J}} b_j\bigg\rangle\bigg\}.$$
		
		Moreover, we define the constraint map $K_i$ for these $I+J+1$ players as follows,
		\begin{equation*}
			K_i(a,b,p)= \begin{cases}
				M_i(p,b),~& i\in\{1,2,\cdots I\}\\
				B_{i-I},~&  i\in \{I+1,I+2,\cdots,I+J\}\\
				\Delta,~& i=I+J+1.
			\end{cases}
		\end{equation*}
		Since $A_i$ and $B_j$ are assumed to be non-empty convex compact sets, the map $K_i$ admits non-empty convex compact values for any $i\in \Lambda$. Further, $K_i$ is lower semi-continuous and closed according to \cite{milasi2013,milasipref}.
		
		One can easily verify that the preference map $P_i$ of any player in the set $\Lambda$ has non-empty convex values and admits open upper sections. Further, $P_i$ is a lower semi-continuous map with $x_i\notin P_i(x)$ for any $x\in A\times B\times \Delta$. 
		
		Hence, the considered GNEP $\Gamma$ admits an equilibrium as per Corollary \ref{QVIcoro}. In the view of Lemma \ref{competitive}, this equilibrium is a required competitive equilibrium for Arrow-Debreu economy under uncertainity.
		
		\end{proof}
		
		If the preferences of players are represented in terms of a binary relation as described in Section \ref{secbinary} and the set $U_i(x)=\{y_i\in X_i|\,(x_{-i},y_i)\succ_i (x_{-i},x_i)\}$ for any $x\in\prod_{i\in \Lambda} X_i$, then we can ensure the occurrence of competitive equilibrium by combining Theorem \ref{comperesult} with Proposition \ref{binary}. 
		\begin{corollary}\label{corol}
			Assume that for any $i\in \mathcal{I}$ and $j\in\mathcal{J}$ we have,
			\begin{itemize}
				\item[(a)]  the sets $A_i\subseteq \mathbb{R}_+^{H}$ and $B_j\subseteq \mathbb{R}^{H}$ are non-empty convex compact;
				\item[(b)] $\succ_i$ is irreflexive, convex, continuous and non-satiated over the set $\prod_{i\in \mathcal{I}} A_i$;
				\item[(c)] $a_i\in cl(U_i(a,b,p))$ for each $(a,b,p)\in A\times B\times \Delta$.
			\end{itemize}
			Then, the Arrow-Debreu economy under uncertainty admits atleast one competitive equilibrium.
		\end{corollary}
		
		\begin{remark}
			One can observe that 
			under the boundedness of consumption sets, the existence result Corollary \ref{corol} derived by us extends \cite[Theorem 3.2]{milasipref} because in \cite{milasipref} the preference relation of involved consumers $\succ_i$ is semi-strictly convex, ordered (complete and transitive) and independent of the price or choice of other agents. We relaxed certain assumptions over $\succ_i$ by considering that the preference relation of any consumer is convex, non-ordered, inter-dependent and price-dependent. 
		\end{remark}
		\section*{Acknowledgement(s)}
		The first author acknowledges Science and Engineering Research Board, India $\big(\rm{MTR/2021/000164}\big)$ for the financial support. The second author is grateful to the University Grants Commission (UGC), New Delhi, India for the financial assistance provided by them throughout this research work under the registration number: $\big(\rm{1313/(CSIRNETJUNE2019)}\big)$. 
		

\begin{thebibliography}{99}
			\bibitem{homidan} Al-Homidan~S, Hadjisavvas~N, Shaalan~L. Transformation of quasiconvex functions to eliminate local minima. J Optim Theory Appl. 2018;177:93--105.
			\bibitem{debreu} Arrow~KJ, Debreu~G. Existence of an equilibrium for a competitive economy. Econometrica. 1954;22:265--290.
			\bibitem{aubin} Aubin~JP. Optima and equilibria: an introduction to nonlinear analysis (2nd Edition). Berlin (NY) Springer; 1998.
			\bibitem{ausselnormal} Aussel~D. New developments in quasiconvex optimization. In: Al-Mezel~SAR, Al-Solamy~FRM, Ansari~QH, editors. Fixed Point Theory, Variational Analysis and Optimization. Boca-Raton (Florida): CRC Press; 2014. 
			\bibitem{ausselGNEP} Aussel~D, Dutta~J. Generalized Nash equilibrium problem, variational inequality and quasiconvexity. Oper Res Lett. 2008;36:461--464.
			\bibitem{aussel-rachna} Aussel~D, Gupta~R, Mehra~A. Evolutionary variational inequality formulation of the generalized Nash equilibrium problem. J Optim Theory Appl. 2016;169:74--90.
			\bibitem{aussel-hadj} Aussel~D, Hadjisavvas~N. On quasimonotone variational inequalities. J Optim Theory Appl. 2004;121:445--450.
			\bibitem{ausselcoer} Aussel~D, Sultana~A. Quasi-variational inequality problems with non-compact valued constraint maps. J Math Anal Appl. 2017;456:1482--1494. 
			\bibitem{bade} Bade~S. Nash equilibrium in games with incomplete preferences. Econ Theory. 2005;26:309–332.
			\bibitem{milasi2013} Benedetti~I, Donato~MB, Milasi~M. Existence for competitive equilibrium by means of generalized quasivariational inequalities. Abstr. Appl. Anal. 2013;1--8.
			\bibitem{boyd} Boyd~S, Vandenberghe~L. Convex Optimization. Cambridge (UK): Cambridge University Press; 2004.
			\bibitem{cotrina} Bueno~O, Carlos~C, Cotrina~J. A note on coupled constraint Nash games. Arxiv preprint, https://doi.org/10.48550/arXiv.2201.04262.
			
			\bibitem{chan} Chan~D, Pang~JS. The generalized quasi-variational inequality problem. Math Oper Res. 1982;7:211--222.
			\bibitem{cotrinaeq} Cotrina~J, Hantoute~A, Svensson~A. Existence of quasi-equilibria on unbounded constraint sets. Optimization. 2022;71:337--354.
			\bibitem{cotrinatime} Cotrina~J, Z{\'u}{\~n}iga~J. Time-dependent generalized Nash equilibrium problem. J. Optim. Theory Appl. 2018;179:1054--1064.
			
			\bibitem{debreu1} Debreu~G. Representation of a preference ordering by a numerical function. In: Thrall~M, Davis~RC, Coombs~CH, editors. Decision Processes. New York (NY): Wiley; 1954.
			\bibitem{debreubook} Debreu G. Theory of value: An axiomatic analysis of economic equilibrium.  New Haven and London: Yale University Press; 1959.
			
			\bibitem{facc} Facchinei~F, Fischer~A, Piccialli~V. On generalized Nash games and variational inequalities. Oper Res Lett. 2007;35:159--164. 
			\bibitem{faccsurvey}  Facchinei~F, Kanzow~C. Generalized Nash equilibrium problems. Ann Oper Res. 2010;175:177--211.
			
			\bibitem{yannelis} He~W, Yannelis~NC. Existence of Walrasian equilibria with discontinuous, non-ordered, interdependent and price-dependent preferences. Econ Theory. 2016;61:497--513.
			\bibitem{kreps} Kreps~DM. A Course in Microeconomics Theory. Princeton (US): Princeton University Press; 1990.
			\bibitem{krepsweak} Kreps~DM. Microeconomic Foundations I: Choice and Competitive Markets. Princeton (US): Princeton University Press; 2013.
			
			\bibitem{milasipref} Milasi~M, Puglis~A, Vitanza~C. On the study of the economic equilibrium problem through preference relations. J Math Anal Appl. 2019;477:153--162. 
			\bibitem{milasipref2021} Milasi~M, Scopelliti~D. A variational approach to the maximization of preferences without numerical representation. J optim theory appl. 2021;190:879--893.
			\bibitem{yannelis2022} Podczeck~K, Yannelis~NC. Existence of Walrasian equilibria with discontinuous, non-ordered, interdependent and price-dependent preferences, without free disposal, and without compact consumption sets. Econ Theory. 2022;73:413--420.
			\bibitem{rockafellar} Rockafellar~RT, Wets~RJB. Variational Analysis. London (NY): Springer Science \& Business Media; 2009.
			\bibitem{rosen} Rosen~JB. Existence and Uniqueness of Equilibrium Points for Concave N-Person Games. Econometrica. 1965;33:520--534.
			\bibitem{shafer} Shafer~W, Sonnenschein~H. Equilibrium in abstract economies without ordered preferences. J Math Econ. 1975;2:345--348. 
			
			\bibitem{stampacchia} Stampacchia~G. Variational inequalities. In: Ghizzetti~A, editor. Theory and applications of monotone operators. Gubbio (Italy): Edizioni Oderisi; 1968.
			\bibitem{tan} Tan~NX. Quasi-variational inequality in topological linear locally convex Hausdorff space. Math Nachr. 1985;122:231--245.
			\bibitem{tian} Tian~G. On the existence of equilibria in generalized games. Int J Game Theory. 1992;20:247--254
			\bibitem{neumann} Von Neumann~J, Morgenstern~O. Theory of Games and Economic Behavior. Princeton (US): Princeton University Press; 1953.
			
			
			
			
			%
			
			
		\end{thebibliography}
	\end{document}